%% file: restriction.tex
\documentclass{amsart}
\usepackage{amsmath,amsfonts,mathrsfs}

\newif\ifdraft
\draftfalse
\input{macros.tex}

\newcommand{\I}{\mathcal{I}}

\newcommand{\Gr}{\operatorname{gr}}

\def\ZZ{{\mathbf Z}}

\def\CC{{\mathbf C}}

\def\QQ{{\mathbf Q}}
\def\PP{{\mathbf P}}

\begin{document}
%\hfill \today

\vspace{\baselineskip}

\title[Restriction theorem for Hodge ideals]{Restriction, subadditivity, and semicontinuity theorems for Hodge ideals}

\author[M. Musta\c{t}\v{a}]{Mircea~Musta\c{t}\u{a}}
\address{Department of Mathematics, University of Michigan,
Ann Arbor, MI 48109, USA}
\email{{\tt mmustata@umich.edu}}

\author[M.~Popa]{Mihnea~Popa}
\address{Department of Mathematics, Northwestern University, 
2033 Sheridan Road, Evanston, IL
60208, USA} \email{{\tt mpopa@math.northwestern.edu}}

\thanks{MM was partially supported by NSF grants DMS-1401227 and DMS-1265256; MP was partially supported by NSF grant
DMS-1405516 and a Simons Fellowship.}


\begin{abstract}
We prove results concerning the behavior of Hodge ideals under restriction to hypersurfaces or fibers of morphisms, and addition. 
The main tool is the description of restriction functors for mixed Hodge modules by means of the $V$-filtration.
\end{abstract}

\maketitle

%\tableofcontents

\section{Introduction}

To any reduced effective divisor $D$ on a smooth complex variety $X$ one can associate a sequence of ideal sheaves
$I_k (D) \subseteq \shO_X$, with $k \ge 0$, called Hodge ideals. They arise from M. Saito's Hodge filtration $F_\bullet$ on  
$\omega_X(*D)$, the sheaf of top forms with arbitrary poles along $D$, seen as the filtered right $\Dmod_X$-module underlying the push-forward of the trivial mixed Hodge module via the open embedding $j \colon U = X \smallsetminus D \hookrightarrow X$. Indeed, it follows from \cite{Saito-B} that one can write 
this filtration as
$$F_k \omega_X (*D) =  \omega_X \big((k+1)D\big) \otimes I_k (D) \,\,\,\,\,\,{\rm for~all} \,\,\,\, k \ge 0.$$
An alternative construction is provided in \cite{MP}. In either approach it is well-understood  that $I_0 (D)$ is a particular example 
of a multiplier ideal, as in \cite{Lazarsfeld}.

The aim of this paper is to establish, in the case of Hodge ideals, the analogues of one of the main properties of multiplier ideals, namely 
the Restriction Theorem \cite[Theorem 9.5.1]{Lazarsfeld}, and of some of its most important consequences. This behavior of restriction to hypersurfaces is exploited further in \cite{MP}.

\begin{intro-theorem}\label{restriction}
Let $X$ be a smooth variety, $D$ a reduced effective divisor on $X$, and $H$ a smooth divisor on $X$ such that 
$H \not\subseteq {\rm Supp}(D)$. If we denote $D_H=D\vert_H$ and $D'_H=(D_H)_{\rm red}$, then for every $k\geq 0$
we have
\begin{equation}\label{eq1_restriction}
\shO_X\big(-(k+1)(D_H-D'_H)\big)\cdot I_k(D'_H)\subseteq I_k(D)\cdot\shO_H.
\end{equation}
In particular, if $D_H$ is reduced, then for every $k\geq 0$ we have
\begin{equation}\label{eq2_restriction}
I_k (D_{H}) \subseteq I_k(D) \cdot \shO_H.
\end{equation}
Moreover, if $H$ is sufficiently general (e.g. a general member of a basepoint-free linear system), then we have equality in (\ref{eq2_restriction})
\end{intro-theorem}

One can easily deduce that the inclusion (\ref{eq2_restriction}) in the theorem still  holds after replacing $H$ with any smooth closed subvariety $Z$
such that $D|_Z$ is reduced; see Corollary \ref{higher_codimension}. For arbitrary multiplier ideals, a restriction theorem to subvarieties
satisfying a natural transversality property with respect to $D$, and further related results, can be found in \cite{DMST}.

We apply Theorem \ref{restriction} in order to establish a relationship between the Hodge ideals of a sum of divisors and those of the individual summands.

\begin{intro-theorem}\label{subadditivity}
If $D_1$ and $D_2$ are effective divisors on the smooth variety $X$, such that $D_1+D_2$ is reduced, then for every $k\geq 0$ 
we have
$$I_k(D_1+D_2)\subseteq\sum_{i+j=k}I_i(D_1)\cdot I_j(D_2)\cdot \shO_X(-jD_1-iD_2)\subseteq I_k(D_1)\cdot I_k(D_2).$$
\end{intro-theorem}

Note that for $k=0$, the assertion is an instance of the Subadditivity Theorem for multiplier ideals; see \cite{DEL}, and also \cite[Theorem~9.5.20]{Lazarsfeld}. 
We will also show that given $D_1$ and $D_2$ as in Theorem~\ref{subadditivity}, we have the inclusions
$$\shO_X\big(-(k+1)D_2\big)\cdot I_k(D_1)\subseteq I_k(D_1+D_2)\subseteq I_k(D_1)$$
for every $k\geq 0$; see Proposition~\ref{behavior_sum} below. The first inclusion follows from the definition of Hodge ideals,
while the second inclusion is a direct consequence of Theorem~\ref{subadditivity}.

Another, this time immediate, consequence of Theorem \ref{restriction}, is a version of inversion of adjunction for $k$-log-canonicity. Recall from \cite{MP} that 
the pair $(X, D)$ is called $k$-log canonical at $x$ if $I_k (D)_x = \shO_{X, x}$; by \cite[Theorem 13.1]{MP} this implies that 
$I_p (D)_x = \shO_{X, x}$ for all $p \le k-1$ as well.

\begin{intro-corollary}\label{inversion}
With the notation of Theorem \ref{restriction}, if for some $x \in H$ the pair $(H, D|_H)$ is $k$-log-canonical at $x$ (hence in particular $D|_H$ is reduced), 
then so is the pair $(X, D)$.
\end{intro-corollary}

A sample application of this fact is a bound on the level of $k$-log-canonicity at a point in terms of the dimension of the singular 
locus of its projective tangent cone; see also \cite[Example 20.5]{MP}.

\begin{intro-corollary}\label{singular_locus_cone}
Let $X$ be a smooth $n$-dimensional variety, $D$ a reduced effective divisor, and $x \in D$ with ${\rm mult}_x (D) = m$ and 
$\dim ~{\rm Sing} \big(\PP (C_x D) \big) = r$. Then 
$$ k \le \frac{n - r - 1}{m} - 1  \,\,\,\, \implies  \,\,\,\, I_k (D)_x = \shO_{X,x},$$
with the convention that $r =-1$ when the projective tangent cone is smooth.
\end{intro-corollary}

Indeed, in the case when $\PP (C_x D)$ is smooth, i.e. for ordinary singularities, the statement is proved (with equality) in 
\cite[Theorem D]{MP}. In general, by taking an affine neighborhood of $x$ with coordinates $x_1, \ldots, x_n$, and considering 
$H$ to be a general linear combination of the $x_i$'s, Corollary \ref{inversion} implies that we can inductively reduce  to the case when the dimension of ${\rm Sing} \big(\PP (C_x D) \big)$ is $r-1$, and therefore conclude. Note that, as explained in \cite[Theorem D]{MP}
and the paragraph preceding it, Corollary \ref{singular_locus_cone} implies that if $X$ is projective and all the singular points of $D$
have multiplicity at most $m$ and projective tangent cone with singular locus of dimension at most $r$, then 
$$F_p H^\bullet (U, \CC) = P_p H^\bullet (U, \CC) \,\,\,\,\,\,{\rm for ~all}\,\,\,\, p \le \frac{n - r - 1}{m} - n -1.$$
Here $F_\bullet$ and $P_\bullet$ are the Hodge and pole order filtration on the singular cohomology of the complement 
$U = X \setminus D$, starting in degree $-n$.

Another useful application of Theorem \ref{restriction} regards the behavior of Hodge ideals in families, in analogy with the 
semicontinuity result for multiplier ideals in \cite[9.5.D]{Lazarsfeld}; this is used in \cite{MP} in order to establish effective nontriviality criteria for Hodge ideals.  

We first fix some notation. Let $h\colon X\to T$ be a smooth morphism of relative dimension $n$
between arbitrary varieties $X$ and $T$, and $s\colon T\to X$ a morphism such that $h\circ s={\rm id}_T$.
Suppose that $D$ is an effective Cartier divisor on $X$, relative over $T$, such that for every $t\in T$ the restriction $D_t$ 
of $D$ to the fiber $X_t=h^{-1}(t)$
is reduced. For every $x\in X$, we denote by $\mathfrak{m}_x$ the ideal defining $x$ in $X_{h(x)}$.

\begin{intro-theorem}\label{semicontinuity}
With the above notation,  for every $q\geq 1$, the set
$$V_q:=\big\{t\in T\mid I_k(D_t)\not\subseteq \mathfrak{m}_{s(t)}^q\big\},$$
is open in $T$. This applies in particular to the set 
$$V_1 = \big\{t\in T\mid (X_t, D_t) {\rm ~is~} k{\rm -log~canonical~at}\,\,s(t) \}.$$
\end{intro-theorem}

Regarding the methods we use, while the emphasis in \cite{MP} is on a birational definition and study of the ideals $I_k(D)$, the approach in this paper relies on their definition and study via mixed Hodge module theory \cite{Saito-MHP}, \cite{Saito-MHM}. In particular, the key technical tools are the Kashiwara-Malgrange $V$-filtration, and the functors of nearby and vanishing cycles. We do not know how to 
obtain the restriction theorem for higher Hodge ideals similarly (or by reduction) to the argument used for multiplier ideals in \cite{Lazarsfeld}.

\medskip

\noindent
{\bf Acknowledgements.}
We thank Christian Schnell for useful suggestions, and especially for help with the global description of the $i^{!}$ functor in \S\ref{shriek}.
We also thank Morihiko Saito and a referee for helpful comments. The second author is grateful to the math department at Stony Brook University for hospitality during the preparation of the paper, and the Simons Foundation for fellowship support.

\section{Background and preliminary results}

For much more background regarding filtered $\Dmod$-modules and mixed Hodge modules, especially in the context of 
Hodge ideals, please see \cite[\S1-10]{MP}. We recall further definitions and facts needed here.

\subsection{$V$-filtration and Hodge filtration}
Let $X$ be a smooth complex variety of dimension $n$ and let $H$ be a smooth divisor on 
$X$. Whenever needed, we choose a local equation $t$ for $H$ and denote by $\partial_t$
a vector field with the property that $[\partial_t,t]=1$. To $H$ one associates an increasing rational filtration\footnote{All rational 
filtrations $F_{\bullet}$ that we consider are discrete, that is, there is a positive integer $d$ such that $F_{\alpha}=F_{i/d}$ whenever 
$\frac{i}{d}\leq \alpha<\frac{i+1}{d}$, with $i\in\ZZ$.}
on the sheaf of differential operators $\Dmod_X$, given by
$$V_{\alpha} \Dmod_X = \{ P \in \Dmod_X~|~ P \cdot \I_{H}^j \subseteq \I_{H}^{j - [\alpha]}\} \,\,\,\,{\rm for~} 
\alpha \in \QQ,$$
where $\I_{H}$ is the ideal of $H$ in $\shO_X$, with the convention that $\I_{H}^j  = \shO_X$ for $j \le 0$.

\begin{definition}[{\cite[3.1.1]{Saito-MHP}}]
Let $\Mmod$ be a coherent right $\Dmod_X$-module. A \emph{Kashiwara-Malgrange $V$-filtration} of $\Mmod$ 
along $H$ is an increasing rational filtration $V_{\bullet}\Mmod$, satisfying the 
following properties:

\noindent
$\bullet$\,\,\,\,The filtration is exhaustive, i.e. $\bigcup_{\alpha} V_{\alpha} \Mmod = \Mmod$, and each $V_{\alpha} \Mmod$
is a coherent $V_0 \Dmod_X$-submodule of $\Mmod$.

\noindent
$\bullet$\,\,\,\,$V_\alpha \Mmod \cdot V_i \Dmod_X \subseteq V_{\alpha + i } \Mmod$ for every $\alpha \in \QQ$ and 
$i \in \ZZ$; furthermore
$$V_{\alpha} \Mmod \cdot t = V_{\alpha-1} \Mmod \,\,\,\, {\rm for ~} \alpha < 0.$$

\noindent
$\bullet$\,\,\,\,The action of $t\partial_t -\alpha$ on $\Gr_{\alpha}^V \Mmod$ is nilpotent for each $\alpha\in\QQ$. Here one defines 
$$\gr_{\alpha}^V \Mmod = V_{\alpha} \Mmod / V_{< \alpha} \Mmod,$$ 
where $V_{< \alpha} \Mmod = \cup_{\beta < \alpha} V_{\beta} \Mmod$.
\end{definition}

Assume now that $\Mmod$ is endowed with a good filtration $F$; see \cite[\S2.1]{HTT}. By definition we put
$$F_p V_{\alpha}\Mmod = F_p \Mmod \cap V_{\alpha} \Mmod$$
and
$$F_p \gr_{\alpha}^V \Mmod = \frac{F_p \Mmod \cap V_{\alpha} \Mmod}{F_p \Mmod \cap V_{< \alpha} \Mmod}.$$

\begin{definition}[{\cite[3.2.1]{Saito-MHP}}]\label{quasi-unipotent}
We say that $(\Mmod, F)$ \emph{admits a rational $V$-filtration along $H$}
if $\Mmod$ has a Kashiwara-Malgrange $V$-filtration along $H$ and the following conditions are satisfied:

\noindent
$\bullet$ \,\,\,\, $(F_p V_{\alpha} \Mmod)\cdot t = F_p V_{\alpha - 1} \Mmod$ \,\, for \,\,$\alpha < 0$.

\noindent
$\bullet$ \,\,\,\, $(F_p \gr_{\alpha}^V \Mmod)\cdot \partial_t = F_{p+1} \gr_{\alpha+1}^V \Mmod$  \,\, for  \,\, 
$\alpha > -1$.

\noindent
Moreover,  $(\Mmod, F)$ is \emph{regular and quasi-unipotent along $H$} if in addition the filtration 
$F_{\bullet} \gr_{\alpha}^V \Mmod$ is a good filtration for $-1 \le \alpha \le 0$.
\end{definition}

If $(\Mmod, F)$ is the filtered $\Dmod$-module underlying a mixed Hodge module, it is known that a rational 
$V$-filtration on $\Mmod$ exists, and is unique. In addition, $(\Mmod, F)$ is required by definition 
to be regular and quasi-unipotent along any smooth divisor $H$ \cite{Saito-MHP}; see also \cite[\S11 and \S12]{Schnell-MHM} for a nice discussion.

Given a filtered $\Dmod$-module $(\Mmod, F)$ on $X$, a sufficiently general hypersurface on $X$ (for instance a general member of a 
basepoint-free linear system) is non-characteristic with respect to $(\Mmod, F)$. For the definition,  see \cite[3.5.1]{Saito-MHP}; we will only use this notion here via the following lemma.

\begin{lemma}[{\cite[Lemme 3.5.6]{Saito-MHP}}]\label{noncharacteristic}
Let $i\colon H \hookrightarrow X$ be the inclusion of a smooth divisor in a smooth complex variety, and let $(\Mmod, F)$ be a filtered coherent right $\Dmod_X$-module for which $H$ is non-characteristic. Then:
\begin{enumerate}
\item $(\Mmod, F)$ is regular and quasi-unipotent along $H$.

\item The $V$-filtration on $\Mmod$ is given by 
$$V_{\alpha} \Mmod = \Mmod \cdot \shO_X (- j H) \,\,\,\,{\rm for~} -j-1 \le \alpha < -j, \,\,j \ge 0 \,\,\,\,{\rm and~} V_{\alpha} \Mmod = \Mmod \,\,
{\rm for~}\alpha \ge 0.$$
\end{enumerate}
\end{lemma}

\subsection{The functor $i^{!}$ for a closed embedding}\label{shriek}
Recall that given a morphism of (not necessarily smooth) complex algebraic varieties $f\colon X\to Y$, we have an induced exact functor
$$f^!\colon {\rm D}^b\MHM(Y)\to {\rm D}^b\MHM(X)$$
between the corresponding derived categories of mixed Hodge modules (see \cite[\S 4.2]{Saito-MHM}). This is by definition the 
right adjoint of the direct image with compact support $f_!$. If $g\colon Y\to Z$ is another morphism, then we have 
\begin{equation}\label{eq_funct_transformation}
f^!\circ g^!\simeq (g\circ f)^!.
\end{equation}

One easy case is when $f$ is an open immersion, when $f^!=f^{-1}$ is the restriction to $U$. 
We will mostly be interested in the case when $f=i$ is a closed embedding. In this case,  if $j\colon Y\smallsetminus X\hookrightarrow Y$ is the inclusion of the complement of $X$,
then for every $M \in {\rm D}^b\MHM(Y)$, we have a functorial distinguished triangle (see \cite[4.4.1]{Saito-MHM}):
\begin{equation}\label{eq_distinguished_trg}
i_*i^! M \longrightarrow M \longrightarrow j_*j^{-1} M \longrightarrow i_*i^! M[1].
\end{equation}
It follows from this exact triangle that if $M\in {\rm D}^b\MHM(Y)$ is such that $j^{-1}M=0$, then the canonical morphism
$i_*i^!M \to M$ is an isomorphism. We have analogous statements for the underlying object in the derived category of filtered 
$\Dmod$-modules.

We record here the following lemma, for future reference.

\begin{lemma}\label{restriction_to_components}
Let $Y$ be a variety with irreducible components $Y_1,\ldots,Y_r$, and for every $k$ with $1\leq k\leq r$, let $i_k\colon Y_k\hookrightarrow Y$ be the inclusion.
If $M \in {\rm D}^b\MHM(Y)$ is such that $i_k^! M =0$ for $2\leq k\leq r$, then the canonical morphism ${i_1}_*i_1^!M \to M$ is an isomorphism.
\end{lemma}

We first prove the following special case:

\begin{lemma}\label{restriction_to_components2}
Let $Y$ be a variety with irreducible components $Y_1,\ldots,Y_r$, and for every $k$ with $1\leq k\leq r$, let $i_k\colon Y_k\hookrightarrow Y$ be the inclusion.
If $M \in D^b\MHM(Y)$ is such that $i_k^! M =0$ for $1\leq k\leq r$, then $M  =  0$.
\end{lemma}

\begin{proof}
We argue by induction on $r$, the case $r=1$ being trivial. For the induction step, let $Y'=Y_2\cup\cdots\cup Y_r$ and $U=Y\smallsetminus Y'$. 
If $i\colon Y'\hookrightarrow Y$ is the inclusion, it follows by induction, the hypothesis, and (\ref{eq_funct_transformation}) that 
$i^! M = 0$. 
On the other hand, if $j\colon U\hookrightarrow Y_1$ is the inclusion, then 
$$(i_1\circ j)^{-1}M \simeq (i_1\circ j)^! M \simeq j^!i_1^! M =  0.$$
By combining these, we deduce from (\ref{eq_distinguished_trg}) that $M= 0$.
\end{proof}

\begin{proof}[Proof of Lemma~\ref{restriction_to_components}]
It follows from (\ref{eq_distinguished_trg}) that it is enough to show that
 if $U=Y\smallsetminus Y_1$ and $j\colon U\hookrightarrow Y$ is the inclusion,
then $j^{-1}M =0$. But note that $U\subseteq Y_2\cup\cdots\cup Y_r$, and if 
$$U\overset{\alpha}\hookrightarrow Y_2\cup\cdots\cup Y_r\overset{\beta}\hookrightarrow Y$$
are the inclusions, then
$$j^! M \simeq \alpha^!\beta^! M  =  0,$$
since $\beta^! M  =  0$ by Lemma~\ref{restriction_to_components2}.
\end{proof}

\medskip

We now specialize to the case when $i\colon X\hookrightarrow Y$ is a closed embedding of smooth varieties, of codimension $1$. 
If $(\Mmod, F)$ is the filtered $\Dmod$-module underlying a 
 mixed Hodge module $M$ on $Y$, then $i^!\Mmod$ has cohomology only in degrees $0$ and $1$, and
${\mathcal H}^0i^!\Mmod$ and ${\mathcal H}^1i^!\Mmod$, with the induced filtrations, underlie mixed Hodge modules on $X$. 
They can be described explicitly in terms of nearby and vanishing cycles corresponding to the $V$-filtration along $X$, as follows. 

Assume first that $X$ is given on an open subset $U\subseteq Y$ by a local equation $t$. It is then well known that on $U$ 
the $\Dmod$-modules ${\mathcal H}^0i^!\Mmod$ and ${\mathcal H}^1i^!\Mmod$ are the cohomologies of the morphism 
$${\rm Var} = \cdot t \colon \Gr^V_0\Mmod \longrightarrow\Gr^V_{-1}\Mmod$$
between the vanishing and nearby cycles along $X$, induced by the action of $t$ on $V_0 \Mmod$; see \cite[2.24]{Saito-MHM}. 

The corresponding global assertion is that 
$${\mathcal H}^0i^!\Mmod\simeq {\rm ker}(\sigma)\quad\text{and}\quad {\mathcal H}^1i^!\Mmod\simeq {\rm coker}(\sigma)$$
for a canonical morphism 
\begin{equation}\label{eq_Var}
\Gr^V_0\Mmod\overset{\sigma}\longrightarrow\Gr^V_{-1}\Mmod\otimes_{\shO_Y}\shO_Y(X),
\end{equation}
with the filtrations induced by the filtrations on $\Mmod$. 
This goes via the specialization functor as in \cite[\S2.2, 2.3]{Saito-MHM}. Alternatively, in this case we can perform a direct 
calculation, which was shown to us by C. Schnell; see also \cite[\S30]{Schnell-MHM} for a related discussion.
Indeed, a local equation $t$ on an open subset $U\subseteq Y$ induces an isomorphism 
$\shO_Y (X)\vert_U\simeq\shO_U$
and via this isomorphism $\sigma$ is given by multiplication by $t$ as above. Consider now another open set $U'$, on which 
$\shO_Y (X)$ is trivialized by another local equation $t'$, and let $g$ be the corresponding transition function of $\shO_Y(X)$ 
on $U \cap U'$, so that $t' = g t$. We thus have a commutative diagram 
$$
\begin{tikzcd}
\Gr^V_0\Mmod \rar{t} \dar{{\rm Id}} & \Gr^V_{-1}\Mmod \dar{g} \\
\Gr^V_0\Mmod \rar{t'}  & \Gr^V_{-1}\Mmod
\end{tikzcd}
$$
which means that the local models glue to give the morphism in ($\ref{eq_Var}$).\footnote{As Schnell points out, it is interesting to note that the cohomologies ${\mathcal H}^0i^!\Mmod$ and ${\mathcal H}^1i^!\Mmod$ have well-defined $\Dmod_X$-module structure, as it can also be checked by a calculation in  local coordinates, while in general $\gr_{\alpha}^V \Mmod$ does not, since the embedding of $\Dmod_X$ into 
$\gr_0^V \Dmod_Y$ depends on the choice of local equation.}

For the examples below, and the following sections, recall that if $X$ is a smooth complex variety of dimension $n$, and $D$ is a reduced 
effective divisor on $X$, then $\big( \omega_X (*D), F \big)$ is the filtered right $\Dmod$-module underlying the mixed Hodge module 
$j_* \QQ_U^H [n]$, where $j \colon U = X\smallsetminus D \hookrightarrow X$ is the inclusion of the complement of $D$, and 
$\QQ_U^H [n]$ is the trivial mixed Hodge module on $U$.

\begin{example}\label{example1}
Let $Y$ be a smooth variety and $E=\sum_{i=1}^rE_i$ a simple normal crossing divisor on $Y$. If $i\colon E_1\hookrightarrow Y$ is the inclusion, then
$$i^!\omega_Y(*E) =  0.$$
 Indeed, it is easy to see using the definition and uniqueness of the $V$-filtration along $E_1$ that 
$$V_{\alpha}\omega_Y(*E)= \omega_Y (*E') \otimes_{\shO_Y}\shO_Y \big(([\alpha] + 1) E_1\big),$$ 
where $E' = \sum_{i=2}^rE_i$.
The morphism $\sigma$ in (\ref{eq_Var}) is an isomorphism, which gives our assertion.
\end{example}

\begin{example}\label{example2}
Let $Y$ and $E$ be as in the previous example. If $E'=\sum_{i=2}^rE_i$, then
$${\mathcal H}^0i^!\omega_Y(*E') = 0\quad\text{and}\quad \big({\mathcal H}^1i^!\omega_Y(*E'),F_{\bullet}\big)\simeq 
\big(\omega_{E_1}(*E'\vert_{E_1}), F_{\bullet+1}\big).$$
Indeed, it is again easy to see  using the definition and uniqueness of the $V$-filtration along $E_1$ that 
$$V_{\alpha}\omega_Y(*E')=\omega_Y(*E')\otimes_{\shO_Y}\shO_Y\big(([\alpha]+1)E_1\big)\,\,\text{for}\,\,\alpha\leq -1,\,\,\text{and}$$
$$V_{\alpha}\omega_Y(*E')=\omega_Y(*E')\,\,\text{for}\,\,\alpha\geq -1$$
(see also Lemma \ref{noncharacteristic}). Therefore
$$\Gr^V_0\omega_Y(*E') =  0\quad\text{and}\quad \Gr^V_{-1}\omega_Y(*E')\otimes_{\shO_Y}\shO_Y(E_1)\simeq\omega_{E_1}(*E'\vert_{E_1}),$$
and the fact that via this isomorphism
$$F_{k-n}\Gr^V_{-1}\omega_Y(*E')\otimes_{\shO_Y}\shO_Y(E_1)\simeq F_{k-n+1}\omega_{E_1}(*E'\vert_{E_1})$$
follows, for example, from the explicit description of the filtrations on $\omega_Y(*E')$ and $\omega_{E_1}(*E'\vert_{E_1})$ in \cite[Proposition~8.2]{MP}.
\end{example}

\begin{remark}
These examples will be used in the proof of Theorem \ref{restriction}. While the use of the $V$-filtration is in the spirit of what follows below (where it is crucially needed for analyzing the $F$-filtration), it is worth noting that the statements hold for basic topological reasons via the Riemann-Hilbert correspondence. For instance, the vanishing in Example \ref{example1} follows by dualizing the obvious identity $i^{-1} j_! \QQ_{Y \smallsetminus E} = 0$. More generally, all of the facts in this section hold at the topological level, using the sheaf-theoretic unipotent 
vanishing and nearby cycles.
\end{remark}

\section{Main results}

\subsection{Restriction theorem}
We use the results of the previous sections to give a proof of the main statement in the paper. In the proof we will use the fact (see for instance \cite[\S6]{MP})
that given a reduced divisor $D$ and a log resolution $f\colon Y \to X$ of the pair $(X, D)$, assumed to be an isomorphism over 
$X\smallsetminus D$, and denoting $E = (f^*D)_{{\rm red}}$, we have a filtered isomorphism
$$f_+ \big( \omega_Y(*E), F \big) \simeq  \big( \omega_X(*D), F \big).$$
Here $f_+ \colon {\bf D}^b \big({\rm FM}(\Dmod_Y)\big) \rightarrow {\bf D}^b \big({\rm FM}(\Dmod_X)\big)$ is the direct image functor 
on the derived category of filtered right $\Dmod_Y$-modules constructed in \cite[\S2.3]{Saito-MHP}.
The $k$-th Hodge ideal $I_k (D)$ satisfies the formula 
$$F_k \omega_X(*D) = \omega_X \big((k+1)D\big) \otimes I_k (D).$$

\begin{proof}[Proof of Theorem \ref{restriction}]
We consider a log resolution $f \colon Y \to X$ of the pair $(X, D+ H)$. Since $H$ is smooth, it follows that $D+H$ has simple normal crossings on $X\smallsetminus D$,
hence we may and will assume that $f$ is an isomorphism over $X\smallsetminus D$.
We put $E = (f^*D)_{\rm red}$. Note that if $\widetilde{H}$ is the strict transform of $H$, then by construction
the divisor $\widetilde{H}+E$ is reduced, with simple normal crossings. Therefore the induced morphism $h\colon \widetilde{H}\to H$ is a log resolution
of $(H,D'_H)$ which is an isomorphism over $H\smallsetminus D'_H$. Moreover,
we have $E\vert_{\widetilde{H}}=(h^*D'_H)_{\rm red}$.

Consider the Cartesian diagram
$$
\begin{tikzcd}
f^{-1}(H) \rar{j} \dar{g} & Y \dar{f} \\
H \rar{i} & X
\end{tikzcd}
$$
where $f^{-1}(H)$ is considered with its reduced structure, with $i$ and $j$ the obvious inclusions. The base change theorem in \cite[4.4.3]{Saito-MHM} says that
\begin{equation}\label{eq_base_change}
i^! f_+ \omega_Y (*E) \simeq g_+ j^ !  \omega_Y(*E).
\end{equation}
Consider also the closed immersion $\alpha\colon \widetilde{H}\hookrightarrow f^{-1}(H)$.

We first show that we have a canonical isomorphism 
\begin{equation}\label{eq_first_isom}
j^!\omega_Y(*E)\simeq \big(\alpha_*\omega_{\widetilde{H}}(*E\vert_{\widetilde{H}}),F_{\bullet+1}\big)[-1].\footnote{Added in revision: Morihiko Saito has 
recently pointed out to us that a substantially quicker proof of this fact can be given by directly applying the base change formula above with $f$ and $g$  replaced 
by  the inclusions of $U$ and $U\cap H$ into $X$ and $H$ respectively; see \cite{Saito-MLCT}.}
\end{equation}
In order to see this, note that if $G$ is an irreducible component of $f^{-1}(H)$ different from $\widetilde{H}$, and $\alpha_G\colon G\hookrightarrow f^{-1}(H)$
is the inclusion, then $G$ is an irreducible component of $E$, hence
$$\alpha_G^!j^!\omega_Y(*E)\simeq (j\circ\alpha_G)^!\omega_Y(*E) = 0$$
by Example~\ref{example1}.
We thus deduce from Lemma~\ref{restriction_to_components} that the canonical morphism
$$\alpha_+\alpha^!j^!\omega_Y(*E)\to j^!\omega_Y(*E)$$
is an isomorphism. Since
$$\alpha_+\alpha^!j^!\omega_Y(*E)\simeq\alpha_+(j\circ\alpha)^!\omega_Y(*E)\simeq\alpha_+\big(\omega_{\widetilde{H}}(*E\vert_{\widetilde{H}}),F_{\bullet+1}\big)[-1]$$
by Example~\ref{example2}, we obtain the isomorphism in (\ref{eq_first_isom}).  We conclude that the right-hand side of (\ref{eq_base_change}) is isomorphic to
$$g_+\alpha_+\big(\omega_{\widetilde{H}}(*E\vert_{\widetilde{H}}),F_{\bullet+1}\big)[-1]\simeq h_+\big(\omega_{\widetilde{H}}(*E\vert_{\widetilde{H}}),F_{\bullet+1}\big)[-1]\simeq
\big(\omega_H(*D'_H),F_{\bullet+1}\big)[-1].$$

We now turn to the left-hand side of (\ref{eq_base_change}) and note as mentioned above that we have a canonical isomorphism
$$ f_+ \omega_Y (*E)\simeq\omega_X(*D).$$
We show that for every $k\geq 0$ there is a canonical morphism 
\begin{equation}\label{eq_morphism}
F_k{\mathcal H}^1i^!\omega_X(*D)\longrightarrow F_{k}\omega_X(*D)\otimes_{\shO_X}\shO_H(H).
\end{equation}

Recall that, according to the discussion in \S\ref{shriek}, if we consider on $\Mmod=\omega_X(*D)$ the $V$-filtration with respect 
to the smooth hypersurface $H$, then we have a morphism
$$
\sigma\colon\Gr^V_0\Mmod\longrightarrow\Gr_{-1}^V\Mmod\otimes_{\shO_X}\shO_X(H)
$$
such that
$${\mathcal H}^1i^!\Mmod \simeq {\rm coker}(\sigma),$$
with 
$F_k{\mathcal H}^1i^!\Mmod$ being isomorphic to the image of 
$$ F_k \gr_{-1}^V \Mmod \otimes_{\shO_X}\shO_X(H)$$
in ${\rm coker}(\sigma)$.
We define a morphism
$$\eta\colon F_k \gr_{-1}^V \Mmod = \frac{F_k  V_{-1}\Mmod}{F_k  V_{<-1}\Mmod}\longrightarrow F_k\Mmod\otimes_{\shO_X}\shO_H$$
by mapping the class of an element $u\in F_k\Mmod\cap V_{-1}\Mmod$ to the class of $u$ in $F_k\Mmod\otimes_{\shO_X}\shO_H$. 
To check that this is well defined, we may restrict to an open subset where $H$ is defined by an equation $t$. Since $\Mmod$ underlies a mixed Hodge module, the first identity in Definition \ref{quasi-unipotent} says that 
$$(F_k V_{\alpha} \Mmod)\cdot t = F_k V_{\alpha - 1} \Mmod\quad\text{for}\quad\alpha < 0.$$
It follows that 
$F_k V_{<-1}\Mmod$ is contained in  $F_k\Mmod\cdot t$, which maps to $0$ in $F_k\Mmod\otimes_{\shO_X}\shO_H$. 
In order to obtain  the canonical morphism (\ref{eq_morphism}), we also need to check that after tensoring with $\shO_X(H)$,
the morphism $\eta$ is $0$ on ${\rm Im}(\sigma)$.
But since $\sigma$ is given by right multiplication by $t$, it is clear that the image of $\sigma$ is mapped to $0$ by
 $\eta\otimes_{\shO_X}\shO_X(H)$.

By putting everything together, we obtain for every $k$ a canonical morphism
$$
F_{k+1}\omega_H(*D'_H)\simeq F_k{\mathcal H}^1g_+ j^!\omega_Y(*E)\simeq F_k{\mathcal H}^1i^!\omega_X(*D)\longrightarrow F_k\omega_X(*D)\otimes_{\shO_X}\shO_H(H).
$$
Applying this with $k$ replaced by $k-n$, where $n=\dim X$, and using the definition of Hodge ideals, we obtain
\begin{align*}
I_k(D'_H)\otimes_{\shO_H}\omega_H\big((k+1)D'_H\big) \to & I_k(D)\otimes_{\shO_X}\omega_X\big((k+1)D\big)\otimes_{\shO_X}\shO_H(H)\\
&\simeq
I_k(D)\otimes_{\shO_X}\omega_H\big((k+1)D_H\big),
\end{align*}
which defines a morphism 
$$
\shO_X\big(-(k+1)(D_H-D'_H)\big)\otimes I_k(D'_H)\longrightarrow I_k(D)\otimes_{\shO_X}\shO_H.
$$
By composing this with the canonical morphism $I_k(D)\otimes_{\shO_X}\shO_H\to I_k(D)\cdot\shO_H$, we obtain a canonical
morphism 
$$
\varphi\colon \shO_X\big(-(k+1)(D_H-D'_H)\big)\cdot  I_k(D'_H)\longrightarrow I_k(D)\cdot\shO_H.
$$
%The first assertion in the theorem follows if we know that $\varphi$ commutes with the inclusions of both sides in $\shO_H$.

Finally, note that the construction of $\varphi$ is compatible with restriction to open subsets. If we restrict to $U=X\smallsetminus D$,
then both the source and the target of $\varphi\vert_U$ are equal to $\shO_U$ and $\varphi_U$ is the identity. This implies that $\varphi$
commutes with the inclusion of the source and the target in the field of rational functions, that is, $\varphi$ is an inclusion map.
This gives the inclusion in (\ref{eq1_restriction}).

In order to prove the last assertion in the theorem, it is enough to show that if $H$ is non-characteristic with respect to $\omega_X(*D)$, then the canonical morphism
(\ref{eq_morphism}) is surjective (note that since $H$ is general, $D_H$ is reduced). 
 This can be checked locally, hence we may assume that we have a local equation $t$ for $H$. In this case, it follows from
Lemma~\ref{noncharacteristic} that
$$\gr^V_0\omega_X(*D)=0\quad\text{and}\quad \gr_{-1}^0\omega_X(*D)=\omega_X(*D)\otimes_{\shO_X}\shO_H.$$
It is then clear that the morphism (\ref{eq_morphism}) is an isomorphism, which completes the proof of the theorem.
\end{proof}

\begin{corollary}\label{higher_codimension}
Let $X$ be a smooth variety and $D$ a reduced effective divisor on $X$. If $Z$ is a smooth closed subvariety of $X$ such that $Z\not\subseteq {\rm Supp}(D)$
and $D\vert_Z$ is reduced, then for every $k\geq 0$ we have
$$I_k(D\vert_Z)\subseteq I_k(D)\cdot\shO_Z.$$
\end{corollary}

\begin{proof}
The assertion is local on $X$, hence after taking a suitable open cover of $X$ we may assume that $Z=H_1\cap\cdots\cap H_r$, the transverse intersection of $r$ smooth divisors, where $r={\rm codim}(Z,X)$. 
Note that after possibly replacing $X$ by a suitable neighborhood of $H$, we may assume that each $Z_i:=H_1\cap\cdots\cap H_i$ is smooth and $D\vert_{Z_i}$ is reduced for every $i$.
The inclusion in the corollary then follows by applying Theorem~\ref{restriction} to deduce that
$$I_k(D\vert_{Z_i})\subseteq I_k(D\vert_{Z_{i-1}})\cdot\shO_{Z_i},$$
for $1\leq i\leq r$ (where $Z_0=X$).
\end{proof}

\begin{remark}\label{rem_restriction_general_hypersurfaces}
Another proof of the second statement in Theorem \ref{restriction} is given in \cite[Theorem 16.1]{MP}. In that proof, the generality assumption on $H$ is explicitly used only to ensure that 
$H$ is smooth, $D\not\subseteq {\rm Supp}(H)$ and $D\vert_H$ is reduced, and given a log resolution $f\colon Y\to X$ of $(X,D)$, it is also a log resolution of $(X,D+H)$ such that $f^*H$ is the strict transform of $H$. 
These conditions also hold if we work simultaneously with several general divisors, which leads to the following more general version. Suppose that 
$D$ is a reduced effective divisor on the smooth $n$-dimensional variety $X$. For every $k\geq 0$, if $H_1,\ldots,H_r$ are general
elements of base-point free linear systems $V_1,\ldots,V_r$ on $X$, and if $Y=H_1\cap\cdots\cap H_r$, then
$$I_k(D\vert_Y)=I_k(D)\cdot\shO_Y.$$
\end{remark}

\subsection{Subadditivity theorem}
This section contains the proof of Theorem~\ref{subadditivity}.  We proceed as in the case of multiplier ideals, by reducing the statement 
to the following proposition, via the Restriction Theorem.

\begin{proposition}\label{product_case}
Let $X_1$ and $X_2$ be smooth varieties and let $D_i$ be a reduced effective divisor on $X_i$, for $i=1,2$. If $B_i={\rm pr}_i^* D_i$,
where ${\rm pr_i}\colon X_1\times X_2\to X_i$ are the canonical projections, then for every $k\geq 0$ we have
$$I_k(B_1+B_2)=\sum_{i+j=k}\big(I_i(D_1)\shO_{X_1}(-jD_1)\cdot\shO_{X_1\times X_2}\big)\cdot \big(I_j(D_2)\shO_{X_2}(-iD_2)\cdot\shO_{X_1\times X_2}\big).$$
\end{proposition}
\begin{proof}
Let $n_i$ be the dimension of $X_i$.
Let $U_i=X_i\smallsetminus D_i$ and $\alpha_i\colon U_i\hookrightarrow X_i$ the inclusion map,
so that $\big(\omega_{X_i}(*D_i),F \big)$ underlies the mixed Hodge module ${\alpha_i}_*\QQ_{U_i}^H [n_i]$. 
Note that if $U=U_1\times U_2$ and $\alpha\colon U\hookrightarrow X$ is the inclusion, then 
$U=X\smallsetminus (B_1+B_2)$, where $X=X_1\times X_2$.
We have a canonical isomorphism
\begin{equation}\label{external}
\alpha_*\QQ_U^{H}[n_1 + n_2]\simeq {\alpha_1}_*\QQ_{U_1}^H[n_1]\boxtimes  {\alpha_2}_*\QQ_{U_2}^H[n_2].
\end{equation}
Indeed,
$$\alpha_*\QQ_U^{H}[n_1+n_2] \simeq \alpha_* \alpha^{-1} \big( \QQ_{X_1}^H[n_1]\boxtimes  \QQ_{X_2}^H[n_2]\big) $$ 
$$\simeq (\alpha_1 \times \alpha_2)_* (\alpha_1\times \alpha_2)^{-1} \big( \QQ_{X_1}^H[n_1]\boxtimes  \QQ_{X_2}^H[n_2]\big) $$
$$\simeq ({\alpha_1}_*\alpha_1^{-1} \times {\rm id}) \circ ({\rm id} \times {\alpha_2}_*\alpha_2^{-1})  \big( \QQ_{X_1}^H[n_1]\boxtimes  \QQ_{X_2}^H[n_2]\big) $$
$$\simeq ({\alpha_1}_*\alpha_1^{-1} \times {\rm id}) \big( \QQ_{X_1}^H[n_1]\boxtimes  {\alpha_2}_* \alpha_2^{-1}\QQ_{X_2}^H[n_2]\big) $$  
$$\simeq {\alpha_1}_* \alpha_1^{-1}\QQ_{X_1}^H[n_1]\boxtimes  {\alpha_2}_* \alpha_2^{-1}\QQ_{X_2}^H[n_2],$$
where the last two isomorphisms follow from \cite[3.8.5]{Saito-MHM}.

The isomorphism in ($\ref{external}$) implies that we have a canonical isomorphism of filtered right $\Dmod_X$-modules
$$\big( \omega_X\big(*(B_1+B_2)\big), F\big)\simeq \big(\omega_{X_1}(*D_1)\boxtimes\omega_{X_2}(*D_2), F\big),$$
where the filtration on the right hand side is the exterior product of the filtrations on the two factors.
Translating this in terms of Hodge ideals, we obtain
$$I_k(B_1+B_2)\cdot\omega_X\big((k+1)(B_1+B_2)\big) = $$
$$=\sum_{i+j=k}\big(I_i(D_1)\cdot \omega_{X_1}((i+1)D_1)\big)\boxtimes\big(I_j(D_2)\cdot \omega_{X_2}((j+1)D_2)\big)$$
and using the fact that 
$\omega_X=\omega_{X_1}\boxtimes\omega_{X_2}$ we deduce the equality in the proposition.
\end{proof}

\begin{proof}[Proof of Theorem~\ref{subadditivity}]
Consider the diagonal embedding $\Delta\colon X\hookrightarrow X\times X$. If we take $B_i={\rm pr}_i^* D_i$ for $i=1,2$, then
$(B_1+B_2)\vert_X=D_1+D_2$, hence
Corollary~\ref{higher_codimension} gives
$$I_k(D_1+D_2)\subseteq I_k(B_1+B_2)\cdot \shO_X.$$
The first inclusion in the theorem then follows from the formula for $I_k(B_1+B_2)$ in Proposition~\ref{product_case}.
The second inclusion is a consequence of the fact that for every $i$ and $j$, we have
$$\shO_X(-jD_1)\cdot I_i(D_1)\subseteq I_{i+j}(D_1)\quad\text{and}\quad\shO_X(-iD_2)\cdot I_j(D_2)\subseteq I_{i+j}(D_2).$$
This in turn is a consequence of the inclusion
$$F_{i-n}\omega_X(*D)\subseteq F_{i+j-n}\omega_X(*D)$$
for any reduced effective divisor $D$ on $X$, and the definition of Hodge ideals.
\end{proof}

\begin{proposition}\label{behavior_sum}
If $D_1$ and $D_2$ are effective divisors on the smooth variety $X$, such that $D_1+D_2$ is reduced, then for every $k\geq 0$ 
we have
$$\shO_X\big(-(k+1)D_2\big)\cdot I_k(D_1)\subseteq I_k(D_1+D_2)\subseteq I_k(D_1).$$
\end{proposition}

\begin{proof}
Let $U=X\smallsetminus (D_1+D_2)$, $V=X\smallsetminus D_1$, and $\gamma=\beta\circ\alpha$, where $\alpha\colon U\hookrightarrow V$ and $\beta\colon V\hookrightarrow X$ are the inclusion maps. 
If $n=\dim X $, it follows from the definition that $\big(\omega_X(*(D_1+D_2)), F\big)$ underlies the mixed Hodge module
$$\gamma_*\QQ_U^{H} [n]=\beta_*\alpha_*\QQ_U^{H}[n].$$
On the other hand, on $V$ we have a morphism of mixed Hodge modules 
$$\QQ_V^{H} [n]\longrightarrow \alpha_*\alpha^{-1}\QQ_V^{H}[n]=\alpha_*\QQ_U^H[n]$$
which induces by push-forward to $X$ a morphism of mixed Hodge modules
$$\beta_*\QQ_V^{H}[n]\longrightarrow \gamma_*\QQ_U^H[n].$$
At the level of filtered right $\Dmod_X$-modules, this corresponds to the natural inclusion
$$\omega_X(*D_1)\hookrightarrow\omega_X\big(*(D_1+D_2)\big).$$
The inclusion between the filtered pieces of level $k-n$ gives
$$I_k(D_1)\otimes\omega_X\big((k+1)D_1\big)\subseteq I_k(D_1+D_2)\otimes \omega_X\big((k+1)(D_1+D_2)\big),$$
which implies the first inclusion in the proposition.
The second inclusion follows directly from Theorem~\ref{subadditivity}.
\end{proof}

\subsection{Semicontinuity theorem}
This section contains the proof of the result on the behavior of Hodge ideals in families.

\begin{proof}[Proof of Theorem \ref{semicontinuity}]
In order to prove the assertion in the theorem, it is enough to prove the following two statements:

\noindent {\bf Claim 1}. The assertion holds when $\dim T =1$.

\noindent{\bf Claim 2}. The set $V_q$ is constructible.

Indeed, if $T\smallsetminus V_q$ were not closed, since it is constructible by Claim 2, there would be an irreducible, $1$-dimensional, locally closed subset $C$ of $T$ and a point $p\in C$, such that $p\in V_q$, but $C\smallsetminus\{p\}$ is contained in $T\smallsetminus V_q$. This contradicts Claim 1, applied to the restriction of $f$ over $C$. 

We begin by proving the first claim, hence we assume that $\dim T =1$. Suppose that $V_q$ is nonempty and let $t_0\in V_q$.
Since $T$ is a curve, it is enough to show that $V_q$ contains an open subset of $T$. 

Let $T'\to T$ be the normalization. If the assertion holds for $X\times_TT'\to T'$, then it clearly also holds for $X\to T$.
Therefore we may and will assume that $T$ (hence also $X$) is smooth. 
If $U$ is an open neighborhood of $s(t_0)$ and $W$ is an open neighborhood of $t_0$, then we may replace $X$ and $T$ by 
$h^{-1}(W\cap s^{-1}(U))$ and $W\cap s^{-1}(U)$,
respectively. Since $s(T)$ is a smooth subvariety of the smooth variety $X$, of codimension $n$, we may assume that the ideal of $s(T)$ in $X$
is generated by $x_{1},\ldots,x_{n}\in\Gamma(X,\shO_X)$. Note that for every $t\in T$, the ideal $\mathfrak{m}_{s(t)}$ is generated by the images of $x_{1},\ldots,x_{n}$
in $\shO_{X_t}$.

By assumption the restriction of $D$ to each $X_t$ is reduced, hence $D$ is reduced as well. 
Since $h$ is smooth, the fiber $X_{t_0}$ is a smooth divisor on $X$. 
Applying Theorem~\ref{restriction}, we conclude that
$$I_k(D_{t_0})\subseteq I_k(D)\cdot\shO_{X_{t_0}}.$$
Since $t_0\in V_q$, we deduce $I_k(D)\cdot\shO_{X_{t_0}}\not\subseteq \mathfrak{m}^q_{s(t_0)}$.
This in turn implies that, after possibly restricting to suitable open subsets, there is a section
$\alpha\in\Gamma \big(X,I_k(D)\big)$ that can be written as $\alpha=Q+P$, with 
$$Q\in (x_{1},\ldots,x_{n})^{q}\quad\text{and}\quad P=\sum_{|u|<q}\varphi_ux^u,$$
where $u=(u_1,\ldots,u_n)$ varies over the tuples of non-negative integers with $|u|=\sum_iu_i<q$,
we have $\varphi_u\in\Gamma(X,\shO_X)$ for every $u$, and $\varphi_{u_0}(s(t_0))\neq 0$ for some $u_0$. 
Let $W_1$ be an open neighborhood of $t_0$ such that $\varphi_{u_0}(s(t))\neq 0$ for every $t\in W_1$.
In this case, we see that for every $t\in W_1$, the section $\alpha\vert_{X_t}$ of $I_k(D)\cdot\shO_{X_t}$
does not lie in $\mathfrak{m}^q_{s(t)}$. 

On the other hand, there 
is an open subset $W_2$ of $T$ such that
$$I_k(D_t)=I_k(D)\cdot\shO_{X_t}\quad\text{for all}\quad t\in W_2.$$
This is a consequence of the second assertion in Theorem~\ref{restriction}.
We then see that the open subset $W_1\cap W_2$ is contained in $V_q$, completing the proof of Claim 1.

We now turn to the proof of Claim 2. 
Arguing by induction on $\dim T$, it is clear that it is enough to show that there is an open subset
$T_0$ of $T$ such that $V_q\cap T_0$ is open. In particular, we are free at any point to replace $T$ by a nonempty open subset.
We may thus assume that $T$ is smooth. More importantly,  we may also use the Generic Flatness theorem: this says that given
any coherent sheaf $\shF$ on $X$,  after replacing $T$
by an open subset, we may assume that $\shF$ is flat over $T$. 

Fixing $k$, after replacing $T$ by an open subset we may assume that 
\begin{equation}\label{eq_restriction_fiber}
I_k(D)\cdot\shO_{X_t}=I_k(D_t)\quad\text{for all}\quad t\in T.
\end{equation}
Indeed, for this we can consider $T$ to have a locally closed embedding in some $\PP^N$.
If $H_1,\ldots,H_r$ are general hyperplanes in $\PP^N$, with $r=\dim T $, then we may apply 
the second assertion in Theorem~\ref{restriction}
(see also Remark~\ref{rem_restriction_general_hypersurfaces}) to 
conclude that
$$I_k(D)\cdot \shO_{L}=I_k(D\vert_L),$$
where $L=h^* H_1 \cap\cdots\cap h^* H_r $.
Since a general fiber of $h$ appears as a disjoint union of connected components of such an $L$, we may assume that (\ref{eq_restriction_fiber}) holds.

Consider now the ideal ${\mathfrak a}\subseteq\shO_X$ defining $s(T)$. By the above, we may assume that both $\shO_X/I_k(D)$ and $\shO_X/{\mathfrak a}^q$ are flat over $T$,
and that the quotient $\shF:=({\mathfrak a}^q+I_k(D))/{\mathfrak a}^q$ is flat over $T$ as well. The former condition implies that
$${\mathfrak a}^q\otimes\shO_{X_t}={\mathfrak m}_{s(t)}^q\quad\text{and}\quad I_k(D)\otimes\shO_{X_t}=I_k(D)\cdot\shO_{X_t}=I_k(D_t),$$
so that the latter condition guarantees that
$$(\shF\otimes\shO_{X_t})_{s(t)}=0\quad\text{if and only if}\quad I_k(D_t)\subseteq {\mathfrak m}_{s(t)}^q.$$
On the other hand, Nakayama's lemma implies that for every $x\in X$, we have 
$$\shF_x=0 \,\, \iff \,\,  (\shF\otimes\shO_{X_{h(x)}})_x  = 0.$$
We thus conclude  that $V_q=s^{-1}\big(X\smallsetminus {\rm Supp}(\shF)\big)$, so that it is open, completing the proof of Claim 2 and consequently of the theorem.
\end{proof}

\end{document}

%% file: macros.tex
\usepackage{amsmath,amsfonts,amsthm,mathrsfs}
\usepackage{amssymb}

\usepackage[unicode,bookmarks]{hyperref}
\usepackage[usenames,dvipsnames]{xcolor}
\hypersetup{colorlinks=true,citecolor=NavyBlue,linkcolor=Brown,urlcolor=Orange}

\usepackage[alphabetic,initials]{amsrefs}

\usepackage{enumitem}

\usepackage{chngcntr}

\ifdraft
\usepackage[notcite,notref,color]{showkeys}

\definecolor{labelkey}{gray}{0.5}
\fi

\usepackage{tikz}

\usepackage{tikz-cd}
\usetikzlibrary{matrix,arrows}
\newlength{\myarrowsize} 

\pgfarrowsdeclare{cmto}{cmto}{
	\pgfsetdash{}{0pt} 
	\pgfsetbeveljoin 
	\pgfsetroundcap 
	\setlength{\myarrowsize}{0.6pt}
	\addtolength{\myarrowsize}{.5\pgflinewidth}
	\pgfarrowsleftextend{-4\myarrowsize-.5\pgflinewidth} 
	\pgfarrowsrightextend{.8\pgflinewidth}
}{
	\setlength{\myarrowsize}{0.6pt} 
  	\addtolength{\myarrowsize}{.5\pgflinewidth}  
	\pgfsetlinewidth{0.5\pgflinewidth}
	\pgfsetroundjoin
	\pgfpathmoveto{\pgfpoint{1.5\pgflinewidth}{0}}
	\pgfpatharc{-109}{-170}{4\myarrowsize}
	\pgfpatharc{10}{189}{0.58\pgflinewidth and 0.2\pgflinewidth}
	\pgfpatharc{-170}{-115}{4\myarrowsize+\pgflinewidth}
	\pgfpathclose
	\pgfusepathqfillstroke
	\pgfpathmoveto{\pgfpoint{1.5\pgflinewidth}{0}}
	\pgfpatharc{109}{170}{4\myarrowsize}
	\pgfpatharc{-10}{-189}{0.58\pgflinewidth and 0.2\pgflinewidth}
	\pgfpatharc{170}{115}{4\myarrowsize+\pgflinewidth}
	\pgfpathclose
	\pgfusepathqfillstroke
	\pgfsetlinewidth{2\pgflinewidth}
}

\pgfarrowsdeclare{cmonto}{cmonto}{
	\pgfsetdash{}{0pt} 
	\pgfsetbeveljoin 
	\pgfsetroundcap 
	\setlength{\myarrowsize}{0.6pt}
	\addtolength{\myarrowsize}{.5\pgflinewidth}
	\pgfarrowsleftextend{-4\myarrowsize-.5\pgflinewidth} 
	\pgfarrowsrightextend{.8\pgflinewidth}
}{
	\setlength{\myarrowsize}{0.6pt} 
  	\addtolength{\myarrowsize}{.5\pgflinewidth}  
	\pgfsetlinewidth{0.5\pgflinewidth}
	\pgfsetroundjoin
	\pgfpathmoveto{\pgfpoint{1.5\pgflinewidth}{0}}
	\pgfpatharc{-109}{-170}{4\myarrowsize}
	\pgfpatharc{10}{189}{0.58\pgflinewidth and 0.2\pgflinewidth}
	\pgfpatharc{-170}{-115}{4\myarrowsize+\pgflinewidth}
	\pgfpathclose
	\pgfusepathqfillstroke
	\pgfpathmoveto{\pgfpoint{1.5\pgflinewidth}{0}}
	\pgfpatharc{109}{170}{4\myarrowsize}
	\pgfpatharc{-10}{-189}{0.58\pgflinewidth and 0.2\pgflinewidth}
	\pgfpatharc{170}{115}{4\myarrowsize+\pgflinewidth}
	\pgfpathclose
	\pgfusepathqfillstroke
	\pgfpathmoveto{\pgfpoint{1.5\pgflinewidth-0.3em}{0}}
	\pgfpatharc{-109}{-170}{4\myarrowsize}
	\pgfpatharc{10}{189}{0.58\pgflinewidth and 0.2\pgflinewidth}
	\pgfpatharc{-170}{-115}{4\myarrowsize+\pgflinewidth}
	\pgfpathclose
	\pgfusepathqfillstroke
	\pgfpathmoveto{\pgfpoint{1.5\pgflinewidth-0.3em}{0}}
	\pgfpatharc{109}{170}{4\myarrowsize}
	\pgfpatharc{-10}{-189}{0.58\pgflinewidth and 0.2\pgflinewidth}
	\pgfpatharc{170}{115}{4\myarrowsize+\pgflinewidth}
	\pgfpathclose
	\pgfusepathqfillstroke
	\pgfsetlinewidth{2\pgflinewidth}
}

\pgfarrowsdeclare{cmhook}{cmhook}{
	\pgfsetdash{}{0pt} 
	\pgfsetbeveljoin 
	\pgfsetroundcap 
	\setlength{\myarrowsize}{0.6pt}
	\addtolength{\myarrowsize}{.5\pgflinewidth}
	\pgfarrowsleftextend{-4\myarrowsize-.5\pgflinewidth} 
	\pgfarrowsrightextend{.8\pgflinewidth}
}{
	\setlength{\myarrowsize}{0.6pt} 
  	\addtolength{\myarrowsize}{.5\pgflinewidth}  
 	\pgfsetdash{}{0pt}
	\pgfsetroundcap
	\pgfpathmoveto{\pgfqpoint{0pt}{-4.667\pgflinewidth}}
	\pgfpathcurveto
    {\pgfqpoint{4\pgflinewidth}{-4.667\pgflinewidth}}
    {\pgfqpoint{4\pgflinewidth}{0pt}}
    {\pgfpointorigin}
	\pgfusepathqstroke
}

\newenvironment{diagram*}[2]{%
\[%
\begin{tikzpicture}[>=cmto,baseline=(current bounding box.center),%
	to/.style={->,font=\scriptsize,cap=round},%
	into/.style={cmhook->,font=\scriptsize,cap=round},%
	onto/.style={-cmonto,font=\scriptsize,cap=round},%
	math/.style={matrix of math nodes, row sep=#2, column sep=#1,%
		text height=1.5ex, text depth=0.25ex}]%
}{%
\end{tikzpicture}%
\]% 
\ignorespacesafterend%
}

\newcommand{\MHM}{\operatorname{MHM}}

\newcommand{\Dmod}{\mathscr{D}}
\newcommand{\Mmod}{\mathcal{M}}

\newcommand{\ZZ}{\mathbb{Z}}
\newcommand{\QQ}{\mathbb{Q}}

\newcommand{\CC}{\mathbb{C}}

\newcommand{\PP}{\mathbb{P}}

\DeclareMathOperator{\gr}{gr}

\newcommand{\shf}[1]{\mathscr{#1}}

\def\overbar#1#2#3{{%
	\setbox0=\hbox{\displaystyle{#1}}%
	\dimen0=\wd0
	\advance\dimen0 by -#2 
	\vbox {\nointerlineskip \moveright #3 \vbox{\hrule height 0.3pt width \dimen0}%
		\nointerlineskip \vskip 1.5pt \box0}%
}}

\newcommand{\shF}{\shf{F}}

\newcommand{\shO}{\shf{O}}

\makeatletter
\let\@@seccntformat\@seccntformat
\renewcommand*{\@seccntformat}[1]{%
  \expandafter\ifx\csname @seccntformat@#1\endcsname\relax
    \expandafter\@@seccntformat
  \else
    \expandafter
      \csname @seccntformat@#1\expandafter\endcsname
  \fi
    {#1}%
}
\newcommand*{\@seccntformat@subsection}[1]{%
  \textbf{\csname the#1\endcsname.}
}
\makeatother

\makeatletter
\let\@paragraph\paragraph
\renewcommand*{\paragraph}[1]{%
	\vspace{0.3\baselineskip}%
	\@paragraph{\textit{#1}}%
}
\makeatother

\counterwithin{equation}{subsection}
\counterwithout{subsection}{section}
\counterwithin{figure}{subsection}

\newtheorem*{theorem*}{Theorem}
\newtheorem{lemma}[equation]{Lemma}
\newtheorem*{lemma*}{Lemma}
\newtheorem{corollary}[equation]{Corollary}
\newtheorem{proposition}[equation]{Proposition}
\newtheorem*{proposition*}{Proposition}

\theoremstyle{definition}
\newtheorem{definition}[equation]{Definition}
\newtheorem*{definition*}{Definition}
\newtheorem{remark}[equation]{Remark}

\newtheorem{example}[equation]{Example}
\newtheorem*{example*}{Example}
\newtheorem*{problem*}{Problem}

\theoremstyle{plain}

\newcommand{\theoremref}[1]{\hyperref[#1]{Theorem~\ref*{#1}}}
\newcommand{\lemmaref}[1]{\hyperref[#1]{Lemma~\ref*{#1}}}
\newcommand{\definitionref}[1]{\hyperref[#1]{Definition~\ref*{#1}}}
\newcommand{\propositionref}[1]{\hyperref[#1]{Proposition~\ref*{#1}}}
\newcommand{\conjectureref}[1]{\hyperref[#1]{Conjecture~\ref*{#1}}}
\newcommand{\corollaryref}[1]{\hyperref[#1]{Corollary~\ref*{#1}}}
\newcommand{\exampleref}[1]{\hyperref[#1]{Example~\ref*{#1}}}

\makeatletter
\let\old@caption\caption
\renewcommand*{\caption}[1]{%
	\setcounter{figure}{\value{equation}}%
	\stepcounter{equation}%
	\old@caption{#1}\relax%
}
\makeatother

\newcounter{intro}

\newtheorem{intro-conjecture}[intro]{Conjecture}
\newtheorem{intro-corollary}[intro]{Corollary}
\newtheorem{intro-theorem}[intro]{Theorem}

\newcommand{\parref}[1]{\hyperref[#1]{\S\ref*{#1}}}

\makeatletter
\newcommand*\if@single[3]{%
  \setbox0\hbox{${\mathaccent"0362{#1}}^H$}%
  \setbox2\hbox{${\mathaccent"0362{\kern0pt#1}}^H$}%
  \ifdim\ht0=\ht2 #3\else #2\fi
  }
\newcommand*\rel@kern[1]{\kern#1\dimexpr\macc@kerna}
\newcommand*\widebar[1]{\@ifnextchar^{{\wide@bar{#1}{0}}}{\wide@bar{#1}{1}}}
\newcommand*\wide@bar[2]{\if@single{#1}{\wide@bar@{#1}{#2}{1}}{\wide@bar@{#1}{#2}{2}}}
\newcommand*\wide@bar@[3]{%
  \begingroup
  \def\mathaccent##1##2{%
    \if#32 \let\macc@nucleus\first@char \fi
    \setbox\z@\hbox{$\macc@style{\macc@nucleus}_{}$}%
    \setbox\tw@\hbox{$\macc@style{\macc@nucleus}{}_{}$}%
    \dimen@\wd\tw@
    \advance\dimen@-\wd\z@
    \divide\dimen@ 3
    \@tempdima\wd\tw@
    \advance\@tempdima-\scriptspace
    \divide\@tempdima 10
    \advance\dimen@-\@tempdima
    \ifdim\dimen@>\z@ \dimen@0pt\fi
    \rel@kern{0.6}\kern-\dimen@
    \if#31
      \overline{\rel@kern{-0.6}\kern\dimen@\macc@nucleus\rel@kern{0.4}\kern\dimen@}%
      \advance\dimen@0.4\dimexpr\macc@kerna
      \let\final@kern#2%
      \ifdim\dimen@<\z@ \let\final@kern1\fi
      \if\final@kern1 \kern-\dimen@\fi
    \else
      \overline{\rel@kern{-0.6}\kern\dimen@#1}%
    \fi
  }%
  \macc@depth\@ne
  \let\math@bgroup\@empty \let\math@egroup\macc@set@skewchar
  \mathsurround\z@ \frozen@everymath{\mathgroup\macc@group\relax}%
  \macc@set@skewchar\relax
  \let\mathaccentV\macc@nested@a
  \if#31
    \macc@nested@a\relax111{#1}%
  \else
    \def\gobble@till@marker##1\endmarker{}%
    \futurelet\first@char\gobble@till@marker#1\endmarker
    \ifcat\noexpand\first@char A\else
      \def\first@char{}%
    \fi
    \macc@nested@a\relax111{\first@char}%
  \fi
  \endgroup
}
\makeatother